\documentclass[a4paper]{amsart}

\usepackage{physics}
\usepackage{amsfonts}
\usepackage{amsthm}
\usepackage{amssymb}
\usepackage{mathrsfs}
\usepackage{geometry}
\usepackage{graphicx}
\usepackage{subfig}
\usepackage{caption}
\usepackage{diagbox}
\usepackage{booktabs}

\DeclareMathOperator{\SPAN}{span}

\DeclareMathOperator{\INT}{int}
\DeclareMathOperator{\BD}{bd}

\DeclareMathOperator{\VOL}{vol}

\renewcommand{\norm}[1]{\left\Vert#1\right\Vert}
\newcommand{\origin}{o}

\newcommand{\DIAM}[1]{\delta\left(#1\right)}

\newcommand{\SET}[2]{\qty{#1\mid #2}}

\newcommand{\PI}{\mathbb{Z}^{+}}

\newcommand{\SCB}[1][n]{\mathcal{K}^{#1}}

\newcommand{\R}{\mathbb{R}}
\newcommand{\E}{\mathbb{E}}

\newtheorem{thm}{Theorem}%[section]
\newtheorem{cor}[thm]{Corollary}
\newtheorem{lem}[thm]{Lemma}
\newtheorem{prop}[thm]{Proposition}
\newtheorem{conj}{Conjecture}
\newtheorem{prob}{Problem}
\theoremstyle{definition}

\newtheorem{rem}[thm]{Remark}

\title{Divide bounded sets into sets having smaller diameters}
\author{Yanlu Lian}
\email{yanlu\_lian@tju.edu.cn}

\author{Senlin Wu}
\email{wusenlin@nuc.edu.cn}

\begin{document}
\maketitle

\begin{abstract}
  For each positive integer $m$ and each real finite dimensional Banach space
  $X$, we set $\beta(X,m)$ to be the infimum of $\delta\in (0,1]$ such that each
  set $A\subset X$ having diameter $1$ can be represented as the union of $m$
  subsets of $A$ whose diameters are at most $\delta$. Elementary properties of
  $\beta(X,m)$, including its stability with respect to $X$ in the sense of
  Banach-Mazur metric, are presented. Two methods for estimating $\beta(X,m)$
  are introduced. The first one estimates $\beta(X,m)$ using the knowledge of
  $\beta(Y,m)$, where $Y$ is a Banach space sufficiently close to $X$. The
  second estimation uses the information about $\beta_X(K,m)$, the infimum of
  $\delta\in(0,1]$ such that $K\subset X$ is the union of $m$ subsets having
  diameters not greater than $\delta$ times the diameter of $K$, for certain
  classes of convex bodies $K$ in $X$. In particular, we show that
  $\beta(l_p^3,8)\leq 0.925$ holds for each $p\in [1,+\infty]$ by applying the
  first method, and we proved that $\beta(X,8)<1$ whenever $X$ is a
  three-dimensional Banach space satisfying $\beta_X(B_X,8)<\frac{221}{328}$,
  where $B_X$ is the unit ball of $X$, by applying the second method. These
  results and methods are closely related to the extension of Borsuk's problem
  in finite dimensional Banach spaces and to C. Zong's computer program for
  Borsuk's conjecture.
\end{abstract}

\section{Introduction}
\label{sec:introduction}
Let $X=(\R^n,\norm{\cdot})$ be an $n$-dimensional Banach space with unit ball
$B_X$. For each $A\subset X$, we denote by $\INT{A}$ and $\BD{A}$ the
\emph{interior} and the \emph{boundary} of $A$, respectively. When $A$ is
nonempty and bounded, we denote by $\DIAM{A}$ the \emph{diameter} of $A$. I.e.,
$\DIAM{A}=\sup\SET{\norm{x-y}}{x,y\in A}$. A compact convex subset of $X$ having
interior points is called a \emph{convex body}. Since any two norms on $\R^n$
induce the same topology, $K$ is a convex body (is bounded, resp.) in $X$ if and
only if it is a convex body (is bounded, resp.) in $\E^n=(\R^n,\norm{\cdot}_E)$,
where $\norm{\cdot}_E$ is the regular Euclidean norm. Let $\mathcal{K}^n$ be the
set of all convex bodies in $\E^n$. A bounded subset $A$ of $X$ is said to be
\emph{complete} if $x\not\in A\Rightarrow \DIAM{(A\cup\qty{x})}>\DIAM{A}$. It is
clear that each complete set is closed and convex. For each bounded subset $A$
of $X$, there always exists a complete set $A^c$, called a \emph{completion} of
$A$, with diameter $\DIAM{A}$ containing $A$. Note that $A$ may have different
completions. Put
\begin{displaymath}
  \mathcal{B}^n=\SET{A\subset \R^n}{\text{$A$ is bounded and }\DIAM{A}>0},\quad \mathcal{C}_X=\SET{A\in \mathcal{B}^n}{A\text{ is complete}}.
\end{displaymath}

In 1933, Borsuk \cite{Borsuk1933} proposed the following:
\begin{prob}[Borsuk's Problem]\label{prob:Borsuk}
  Is it possible to divide every bounded subset of $\E^n$ into $n+1$ sets having
  smaller diameters?
\end{prob}

The answer is affirmative when $n\leq 3$ (cf. \cite{Eggleston1955},
\cite{Grunbaum1957}, \cite{Heppes-Revesz1956}, or \cite{Heppes1957}), is
negative in each of the following cases: $n\geq 298$ (cf.
\cite{Hinrichs-Richter2003}), $n\geq 65$ (cf. \cite{Bondarenko2014}), and $n=64$
(cf. \cite{Jenrich-Brouwer2014}), and is not clear in other cases.

Progesses have been made by providing upper bounds for
$b(n)=\sup\SET{b(A)}{A\in\mathcal{B}^n}$, where $b(A)$, called the \emph{Borsuk
  number of $A$}, is the minimal positive integer $m$ such that $A$ is the union
of $m$ subsets having smaller diameter. For example, L. Danzer (cf.
\cite{Danzer1961}), M. Lassak (cf. \cite{Lassak1982}), and O. Schramm (cf.
\cite{Schramm1988} and \cite{Bourgain-Lindenstrauss1991}) showed that
\begin{displaymath}
  b(n)\leq\sqrt{\frac{(n+2)^3(2+\sqrt{2})^{n-1}}{3}},\quad b(n)\leq
  2^{n-1}+1,\qqtext{and} b(n)\leq 5n^{\frac{3}{2}}(4+\log n)\qty(\frac{3}{2})^{\frac{n}{2}},
\end{displaymath}
respectively.

One can also study Problem \ref{prob:Borsuk} via estimating
$\beta(n,m)=\sup\SET{\beta(A,m)}{A\in\mathcal{B}^n}$, where $m$ is a positive
integer and, for each $A\in\mathcal{B}^n$,
\begin{displaymath}
  \beta(A,m)=\inf\SET{\frac{1}{\DIAM{A}}\max\limits_{k\in[m]}\DIAM{A_k}}{A=\bigcup\limits_{k=1}^m
  A_k}.
\end{displaymath}
Here we used the shorthand notation $[m]:=\SET{i\in\PI}{1\leq i\leq m}$.  

In 2020, C. Zong (cf. \cite{zong2020}) proved that $\beta(A,m)$ is uniformly continuous on the space
$\mathcal{K}^n$ endowed with the Hausdorff metric, and reformulated Problem \ref{prob:Borsuk} as the
following:
\begin{prob}
  Does there exists a positive number $\alpha_n<1$ such that
  \begin{displaymath}
    \beta(K,n+1)\leq \alpha_n,~\forall K\in\mathcal{D}^n,
  \end{displaymath}
  where
  \begin{displaymath}
    \mathcal{D}^n=\SET{K\in\mathcal{K}^n}{B_2^n\subseteq K\subseteq \frac{\sqrt{n/(2n+2)}}{1-\sqrt{n/(2n+2)}}B_2^n}?
  \end{displaymath}
\end{prob}

Some known estimations of $\beta(n,m)$ are listed in Table \ref{tab:estimations}
below (cf., \cite{Grunbaum1957}, \cite{Kupavskii-Raigorodskii2010},
\cite{Filimonov2010}, and \cite{Filimonov2014}).

\begin{table}[h]
  \centering
  \begin{tabular}{c|c|c|c|c|c|c}
    \toprule
    \diagbox{$n$}{$\beta(n,m)$}{$m$}& $3$ & $4$ & $5$ & $6$ & $7$ & $8$\\
    \midrule
    $2$ & $\frac{\sqrt{3}}{2}$ & $\frac{\sqrt{2}}{2}$ & $\geq \sin\frac{\pi}{5}$ & $\leq \frac{\sqrt{3}}{3}$ & $\frac{1}{2}$ & ?\\
    \midrule
    $3$ & $-$ &$\in[0.8880,0.9887]$ & $\in\qty[\frac{\sqrt3}{2},0.9425]$ & $\geq \frac{\sqrt{6}}{3}$ & $\geq \frac{\sqrt{2}}{2}$ & $\geq \frac{\sqrt{2}}{2}$\\
    \bottomrule
  \end{tabular}
  \caption{Known estimations on $\beta(n,m)$}
  \label{tab:estimations}
\end{table}

Gr\"unbaum extended Borsuk's problem into Banach spaces and asked the following (cf. \cite{Grunbaum1957-1958}):
\begin{prob}\label{prob:Banach-version}
  Let $A\subset X=(\R^n,\norm{\cdot})$. What is the smallest positive integer
  $m$, denoted by $b_X(A)$, such that $A$ can be represented as the union of $m$ sets having smaller diameter.
\end{prob}
Put
\begin{gather*}
  b(X)=\max\SET{b_X(A)}{A\in\mathcal{B}^n}=\max\SET{b_X(A)}{A\in\mathcal{K}^n},\\
  B(n)=\max\SET{b((\R^n,\norm{\cdot}))}{\norm{\cdot}\text{ is a norm on $\R^n$}}.
\end{gather*}
It is clear that
\begin{equation}
  \label{eq:covering-number-as-upperbound}
  b_X(K)\leq c(K),~\forall K\in\mathcal{K}^n,
\end{equation}
where $c(K)$ is the least number of smaller homothetic copies of $K$ needed to
cover $K$. Since Hadwiger's covering conjecture (see, e.g.,
\cite{Boltyanski-Martini-Soltan1997}, \cite{Martini-Soltan1999},
\cite{Brass-Moser-Pach2005}, \cite{Zong2010}, \cite{Bezdek-Khan2018}) asserts
that $c(K)\leq 2^n,~\forall K\in \mathcal{K}^n$, it is reasonable to make the
following conjecture (cf. \cite[p. 75]{Boltyanski-Gohberg1985}):
\begin{conj}\label{conj:normed}
  For each integer $n\geq 3$, $B(n)=2^n$.
\end{conj}

When $X$ is two-dimensional and $A\in\mathcal{B}^2$, $b_X(A)\in\qty{2,3,4}$ (cf.
\cite[\S 33]{Boltyanski-Martini-Soltan1997}). Therefore, $B(2)=4$. L. Yu and C.
Zong \cite{Yu-Zong2009} proved that
\begin{equation}\label{eq:Yu-Zong}
  b(l_p^3)\leq 8,~\forall p\in[1,+\infty].
\end{equation}
By the main result of \cite{Huang-Slomka-Tkocz-Vritsiou2020} and
\eqref{eq:covering-number-as-upperbound}, there exist universal constants $c_1$
and $c_2 > 0$ such that $B(n)\leq c_1 4^n e^{-c_2\sqrt{n}},~\forall n\geq 2$.
Despite of these progress, Conjecture \ref{conj:normed} is still open when
$n\geq 3$.

\bigskip

In this paper, we study Conjecture \ref{conj:normed} by estimating
\begin{displaymath}
  \beta_X(A,m)=\inf\SET{\frac{1}{\DIAM{A}}\max\SET{\DIAM{A_k}}{k\in[m]}}{A=\bigcup\limits_{k=1}^m A_k}
\end{displaymath}
for $A\in\mathcal{B}^n$, and
\begin{displaymath}
  \beta(X,m)=\sup\SET{\beta_X(A,m)}{A\in\mathcal{B}^n}.
\end{displaymath}
We focus mainly, but not only, on the case $n=3$.

In Section \ref{sec:elementary-prop}, we present elementary properties of
$\beta(X,m)$, including its stability with respect to $X$ in the sense of
Banach-Mazur metric. In Section \ref{sec:beta-space-via-ball}, we provide an
estimation of $\beta(X,8)$ for three-dimensional Banach spaces $X$ such that
$\beta_X(B_X,8)$ is sufficiently small. In Section \ref{sec:lp}, we show that
\begin{displaymath}
  \beta(l_p^3,8)\leq 0.925,~\forall p\in [1,+\infty],
\end{displaymath}
which can be viewed as a quantitative version of Yu and Zong's result \eqref{eq:Yu-Zong}.

\section{Elementary properties of $\beta(X,m)$}
\label{sec:elementary-prop}

\begin{prop}
  For each finite dimensional Banach space $X=(\R^n,\norm{\cdot})$ and each positive integer $m$, we
  have
  \begin{displaymath}
    \beta(X,m)=\sup\SET{\beta_X(A,m)}{A\in\mathcal{C}_X}.
  \end{displaymath}
\end{prop}
\begin{proof}
  Put $\beta=\sup\SET{\beta_X(A,m)}{A\in\mathcal{C}_X}$. We only need to show
  that $\beta(X,m)\leq \beta$. Let $A$ be an arbitrary set in $\mathcal{B}^n$,
  and $A^c$ be a completion of $A$. For each $\varepsilon>0$, there exists a
  collection $\SET{B_i}{i\in[m]}$ of subsets of $A^c$ such that
  \begin{displaymath}
    A^c=\bigcup\limits_{i\in[m]} B_i,\qqtext{and} \frac{1}{\DIAM{A}}\max\SET{\delta_X(B_k)}{k\in[m]}\leq \beta_X(A^c,m)+\varepsilon.
  \end{displaymath}
  It follows that
  \begin{displaymath}
    A=A\cap A^c=\bigcup\limits_{i\in[m]} \qty(A\cap B_i).
  \end{displaymath}
  Thus
  \begin{displaymath}
    \beta_X(A,m)\leq \frac{1}{\DIAM{A}}\max\SET{\DIAM{A\cap B_k}}{k\in[m]}\leq \beta_X(A^c,m)+\varepsilon,
  \end{displaymath}
  which implies that $\beta_X(A,m)\leq \beta_X(A^c,m)$. Thus $\beta(X,m)\leq
  \beta$ as claimed.
\end{proof}

Let $\mathcal{T}^n$ be the set of all non-singular linear transformations on
$\R^n$. The (multiplicative) Banach-Mazur metric $d_{BM}^M:~\mathcal{K}^n
\times \mathcal{K}^n \mapsto \mathbb{R}$ is defined by
\begin{displaymath}
  d_{BM}^M(K_1, K_2)=\inf\SET{\gamma\geq 1}{\exists T\in \mathcal{T}^n, v\in
    \mathbb{R}^n \text{ s.t. } T(K_2)\subset K_1 \subset \gamma T(K_2)+v},~\forall K_1,K_2\in\SCB.
\end{displaymath}
The infimum is clearly attained. When both $K_1$ and $K_2$ are symmetric with
respect to $\origin$, we have
\begin{displaymath}
  d_{BM}^M(K_1, K_2)=\inf\SET{\gamma\geq 1}{\exists T\in \mathcal{T}^n\text{
      s.t. } T(K_2)\subset K_1 \subset \gamma T(K_2)}.
\end{displaymath}
In this situation, $d_{BM}^M(K_1, K_2)$ equals to the Banach-Mazur distance
between the Banach spaces $X$ and $Y$ having $K_1$ and $K_2$ as unit balls,
respectively. I.e.,
\begin{displaymath}
  d_{BM}^M(K_1, K_2)=d_{BM}^M(X, Y):=\inf\SET{\norm{T}\cdot\norm{T^{-1}}}{T\text{ is an
      isomorphism from $X$ onto $Y$}}.
\end{displaymath}

We have the following result showing the stability of $\beta(X,m)$ with respect
to $X$ in the sense of Banach-Mazur metric.
\begin{thm}\label{thm:stability}
  If $X=(\R^n,\norm{\cdot}_X)$ and $Y=(\R^n,\norm{\cdot}_Y)$ are two Banach
  space satisfying $d_{BM}^M(X,Y)\leq \gamma$ for some $\gamma\geq 1$, then
  \begin{displaymath}
    \beta(X,m)\leq \gamma\beta(Y,m),~\forall m\in\PI.
  \end{displaymath}
\end{thm}
\begin{proof}
  By applying a suitable linear transformation if necessary, we may assume that
  \begin{displaymath}
    B_Y\subseteq B_X\subseteq\gamma B_Y.
  \end{displaymath}
  In this situation we have, for each $x\in\R^n$,
  \begin{align*}
    \norm{x}_X=\inf\SET{\lambda>0}{x\in\lambda B_X}&\leq \inf\SET{\lambda>0}{x\in\lambda B_Y}\\
                                                   &=\norm{x}_Y\\
                                                   &=\inf\SET{\lambda\gamma>0}{x\in\lambda \gamma B_Y}\\
                                                   &=\gamma\inf\SET{\lambda>0}{x\in\lambda\gamma B_Y}\\
                                                   &\leq\gamma\inf\SET{\lambda>0}{x\in\lambda B_X}=\gamma\norm{x}_X.
  \end{align*}
  Hence,
  \begin{equation}
    \label{eq:compare-norms}
    \norm{x}_X\leq \norm{x}_Y\leq \gamma\norm{x}_X.
  \end{equation}
  
  In the rest of this proof, we denote by $\delta_X(A)$ and $\delta_Y(A)$ the
  diameter of a bounded subset $A$ of $\R^n$ with respect to $\norm{\cdot}_X$
  and $\norm{\cdot}_Y$, respectively. By \eqref{eq:compare-norms}, we have
  \begin{equation}
    \label{eq:compare-diameters}
    \delta_X(A)\leq \delta_Y(A)\leq \gamma\delta_X(A),~\forall A\in\mathcal{B}^n.
  \end{equation}

  Let $A$ be a bounded subset of $X$. Then $A$ is also bounded in $Y$. Let $A^c$
  be a completion of $A$ in $Y$. For any $\varepsilon>0$, there exists a
  collection $\SET{B_i}{i\in[m]}$ of subsets of $A^c$ such that $A^c$ is the
  union of this collection and that
  \begin{displaymath}
    \frac{1}{\delta_Y(A^c)}\max\SET{\delta_Y(B_i)}{i\in[m]}\leq \beta_Y(A^c,m)+\varepsilon.
  \end{displaymath}
  Then
  \begin{align*}
    A=A\cap A^c=\bigcup\limits_{i\in[m]}(B_i\cap A),
  \end{align*}
  and, by \eqref{eq:compare-diameters},  
  \begin{align*}
    \frac{1}{\delta_X(A)}\max\SET{\delta_X(B_i\cap A)}{i\in[m]}&\leq \frac{\gamma}{\delta_Y(A)}\max\SET{\delta_Y(B_i)}{i\in[m]}\\
                                                            &=\frac{\gamma}{\delta_Y\qty(A^c)}\max\SET{\delta_Y(B_i)}{i\in[m]}  \\
                                                            &\leq \gamma(\beta_Y(A^c,m)+\varepsilon).
  \end{align*}
  Therefore, $\beta_X(A,m)\leq \gamma\beta_Y(A^c,m)$. It follows that $\beta(X,m)\leq \gamma\beta(Y,m)$.  
\end{proof}

\begin{cor}
  If $X=(\R^n,\norm{\cdot}_X)$ and $Y=(\R^n,\norm{\cdot}_Y)$ are isometric, then $\beta(X,m)=\beta(Y,m)$.
\end{cor}

\begin{prop}
  Let $l_\infty^n=(\R^n,\norm{\cdot}_\infty)$. Then $\beta(l_\infty^n,2^n)=\frac{1}{2}$.
\end{prop}
\begin{proof}
Put $X=l_\infty^n$. Then every complete set in $X$ is a homothetic copy of $B_X$, see \cite{Eggleston1965} and \cite{Soltan1977}. Therefore,
\begin{displaymath}
  \beta_X(A,2^n)=\beta_X(B_X,2^n),~\forall A\in\mathcal{C}_X.
\end{displaymath}
Thus it sufficies to show $\beta_X(B_X,2^n)=\frac{1}{2}$.

On the one hand, $B_X=\frac{1}{2}B_X+\frac{1}{2}V$, where $V$ is the set of vertices of $B_X$. Since the cardinality of $V$ is $2^n$, we have $\beta_X(B_X,2^n)\leq \frac{1}{2}$.

On the other hand, suppose that $B_X$ is the union of $2^n$ of its subsets $B_1,\cdots,B_{2^n}$. For each $i\in[2^n]$, let $B_i^c$ be a completion of $B_i$. Then
\begin{displaymath}
  B_X\subseteq\bigcup\limits_{i\in[2^n]}B_i^c.
\end{displaymath}
It follows that
\begin{displaymath}
  \VOL{B_X}\leq \sum\limits_{i\in[2^n]}\VOL{B_i^c}\leq \sum\limits_{i\in[2^n]}\qty(\frac{1}{2^n})\max\SET{\delta_X(B_i^c)}{i\in[2^n]}\VOL{B_X},
\end{displaymath}
which implies that $\max\SET{\delta_X(B_i^c)}{i\in[2^n]}\geq 1$. Thus
$\beta_X(B_X,2^n)\geq \frac{1}{2}$, which completes the proof.
\end{proof}

\begin{cor}
  Let $X=(\R^n,\norm{\cdot})$. If $d_{BM}^M(X,l_\infty^n)<2$, then $\beta(X,2^n)<1$.
\end{cor}

We end this section with the following result.
\begin{prop}
$\sup\SET{\beta(X,4)}{\text{X is a two-dimensional Banach space}}=\frac{\sqrt{2}}{2}$.
\end{prop}
\begin{proof}
  Put $\eta=\sup\SET{\beta(X,4)}{X\text{ is a two-dimensional Banach space}}$.
  Let $K\subset\R^2$ be a planar convex body. By the main result in
  \cite{Lassak1986}, $K$ can be covered by four translates of
  $\frac{\sqrt{2}}{2}K$. It follows that $\beta_X(K,4)\leq \frac{\sqrt2}{2}$
  holds for each two-dimensional Banach space $X$. Thus, $\eta\leq
  \frac{\sqrt2}{2}$.

  Let $X=l_2^2$ and $B_X$ be the unit disk of $l_2^2$. To show that $\eta\geq
  \frac{\sqrt2}{2}$, we only need to prove $\beta_X(B_X,4)\geq
  \frac{\sqrt{2}}{2}$. Suppose the contrary that $B_X$ is the union of $A_1,
  A_2, A_3, A_4$, where
  \begin{displaymath}
    \max\SET{\DIAM{A_i}}{i\in[4]}<\frac{\sqrt{2}}{2}.
  \end{displaymath}
  Let $v_1$, $v_2$, $v_3$, and $v_4$ be the vertices of a square
  inscribed in the unit circle $S_X$ of $X$. Then for any $\qty{i,j}\subset[4]$
  \begin{displaymath}
    \norm{v_i-v_j}\geq\sqrt{2}.
  \end{displaymath}
  Assume without loss of generality  that $v_1\in A_1$ and $v_2\in A_2,$ then
  \begin{displaymath}
    \frac{v_1+v_2}{\norm{v_1+v_2}}\notin A_1\cup A_2.
  \end{displaymath}
  Assume that $\frac{v_1+v_2}{\norm{v_1+v_2}}\in A_3$. Then $v_3, v_4\notin
  A_1\cup A_2\cup A_3$, and $A_4$ cannot contain both $v_3$ and $v_4$, a
  contradiction. Thus, $\beta_X(B_X,4)\geq \frac{\sqrt{2}}{2}$ as claimed.
\end{proof}

\section{Estimating $\beta(X,m)$ via $\beta_X(B_X,m)$}
\label{sec:beta-space-via-ball}
Let $S$ be a simplex in $X=(\R^n,\norm{\cdot})$. If the distance between each
pair of vertices of $S$ all equals to $\DIAM{S}$, then we say that $S$ is \emph{equilateral}.

\begin{prop}
  Let $T$ be a triangle in $\R^2$ and $X=(\R^2,\norm{\cdot})$. Then
  $\beta_X(T,4)\leq \frac{1}{2}$. If $T$ is equilateral in $X$, then $\beta_X(T,4)=\frac{1}{2}$.
\end{prop}
\begin{proof}
  We only need to consider the case when $T$ is equilateral. Assume without loss
  of generality that $\DIAM{T}=1$. It is clear that
  $\beta_X(T,4)\leq \frac{1}{2}$. Denote by $\qty{a,b,c}$ the set of vertices of
  $T$, and by $\qty{p,q,r}$ the set of midpoints of three sides of $T$. Then
  $\norm{p-q}=\norm{p-r}=\norm{q-r}=\frac{1}{2}$.

  Suppose the contrary that $\beta_X(T,4)<\frac{1}{2}$. Then there exist four
  subsets $T_1,T_2,T_3,T_4$ of $T$ such that $T=\bigcup_{i\in[4]}T_i$ and that
  $\max\SET{\DIAM{T_i}}{i\in[4]}<\frac{1}{2}$. We may assume that $a\in T_1$,
  $b\in T_2$, and $c\in T_3$. Since
  $\qty{p,q,r}\cap\bigcup_{i\in[3]}T_i=\emptyset$, we have $\qty{p,q,r}\subseteq
  T_4$, which is impossible. Thus $\beta_X(T,4)=\frac{1}{2}$ as claimed.
\end{proof}

\begin{prop}
  \label{prop:simplex}
  Let $T$ be a simplex in $X=(\R^3,\norm{\cdot})$. Then
  \begin{displaymath}
    \beta_X(T,8)\leq \frac{9}{16}.
  \end{displaymath}
\end{prop}
\begin{proof}
  Denote by $\SET{v_i}{i\in[4]}$ the set of vertices of $T$. Without loss of
  generality we may assume that $\origin=\frac{1}{4}\sum_{i\in[4]}v_i$. For each
  $i\in[4]$, put $T_i=\frac{7}{16}v_i+\frac{9}{16}T$. Then the portion of $T$
  not covered by $\bigcup_{i\in[4]}T_i$ is
  \begin{displaymath}
    T_5=\SET{\sum_{i\in[4]}\lambda_iv_i}{\lambda_i\in\qty[0,\frac{7}{16}],~\forall
      i\in [4],~\sum_{i\in[4]}\lambda_i=1}.
  \end{displaymath}
  Suppose that $\sum_{i\in[4]}\lambda_iv_i\in T_5$. Then
  \begin{align*}
    \frac{4}{7}\sum_{i\in[4]}\lambda_iv_i-\origin&=\frac{4}{7}\sum_{i\in[4]}\lambda_iv_i-\frac{1}{4}\sum_{i\in[4]}v_i\\
                                                 &=-\sum_{i\in[4]}\qty(\frac{1}{4}-\frac{4}{7}\lambda_i)v_i\\
                                                 &=-\qty(\sum_{j\in[4]}\qty(\frac{1}{4}-\frac{4}{7}\lambda_j))\sum_{i\in[4]}\frac{\frac{1}{4}-\frac{4}{7}\lambda_i}{\sum_{j\in[4]}(\frac{1}{4}-\frac{4}{7}\lambda_j)}v_i\\
    &\in-\frac{3}{7}T.
  \end{align*}
  It follows that $\sum_{i\in[4]}\lambda_iv_i\in -\frac{3}{4}T$. Thus
  $T_5\subset -\frac{3}{4}T$. It is not difficult to verify that $T$ can be
  covered by $4$ translates of $\frac{3}{4}T$, which implies that $T_5$ can be
  covered by $4$ translates of $-\frac{9}{16}T$. Therefore, $\beta_X(T,8)\leq \frac{9}{16}$.
\end{proof}

\begin{prop}\label{prop:simplex-9}
  Let $T$ be a simplex in $X=(\R^3,\norm{\cdot})$. Then
  \begin{displaymath}
    \beta_X(T,9)\leq \frac{9}{17}.
  \end{displaymath}
\end{prop}
\begin{proof}
  We use the idea in the proof of Proposition \ref{prop:simplex}. For each
  $i\in[4]$, put $T_i=\frac{8}{17}v_i+\frac{9}{17}T$. Then the portion of $T$
  not covered by $\bigcup_{i\in[4]}T_i$ is
  \begin{displaymath}
    T_5=\SET{\sum_{i\in[4]}\lambda_iv_i}{\lambda_i\in\qty[0,\frac{8}{17}],~\forall
    i\in[4],~\sum_{i\in[4]}\lambda_i=1}.
\end{displaymath}
As in the proof of Proposition \ref{prop:simplex}, $T_5\subset -\frac{15}{17}T$.
By using the idea in the proof of Proposition \ref{prop:simplex} again, one can
show that $\beta_X(T,5)\leq \frac{3}{5}$. Therefore, $T_5$ is the union of $5$
subsets of $T_5$ whose diameters are not larger than $\frac{9}{17}\DIAM{T}$. It
follows that $\beta_X(T,9)\leq \frac{9}{17}$.
\end{proof}

\begin{rem}
  The estimations in Proposition \ref{prop:simplex} and Proposition
  \ref{prop:simplex-9} are independent of the choice of norm on $\R^3$.
\end{rem}

For a convex body $K$, \emph{the Minkowski measure of symmetry}, denoted by
$s(K)$, is defined as 
\begin{displaymath}
  s(K)=\min\SET{\lambda >0}{\exists x\in X\text{ s.t. } K+x\subset -\lambda K}.
\end{displaymath}
It is known that
\begin{displaymath}
  1\leq s(K) \leq n,~\forall  K\in\SCB;
\end{displaymath}
the equality on the left holds if and only if $K$ is centrally symmetric, and
the equality on the right holds if and only if $K$ is a simplex (cf. \cite{Toth2015}).

The following lemma shows the stability of $\beta_X(K,m)$ with respect to $K$.
\begin{lem}\label{lem:estimation-via-Hausdorff}
  Let $X=(\R^n,\norm{\cdot})$, and $K$ and $L$ be two convex bodies in $X$. If
  there exist a number $\gamma\geq 1$ and a point $c\in \R^n$ such that
  \begin{displaymath}
    K\subseteq L\subseteq \gamma K+c,
  \end{displaymath}
  then, for each $m\in\PI$, we have
  \begin{displaymath}
    \beta_X(L,m)\leq \gamma\beta_X(K,m).
  \end{displaymath}
\end{lem}
\begin{proof}
  For each $\varepsilon>0$, there exists a collection $\SET{K_i}{i\in[m]}$ of
  subsets of $\gamma K+c$ such that
  \begin{displaymath}
    \gamma K+c=\bigcup_{i\in[m]}K_i
  \end{displaymath}
  and
  \begin{displaymath}
    \DIAM{K_i}\leq\gamma\DIAM{K}\beta_X(K,m)+\varepsilon\leq \gamma\DIAM{L}\beta_X(K,m)+\varepsilon,~\forall i\in[m].
  \end{displaymath}
  Since
  \begin{displaymath}
    L=L\cap (\gamma K+c)=\bigcup_{i\in[m]}\qty(L\cap K_i),
  \end{displaymath}
  we have
  \begin{displaymath}
    \beta_X(L,m)\leq\gamma\beta_X(K,m)+\frac{\varepsilon}{\DIAM{L}}.
  \end{displaymath}
  Since $\varepsilon$ is arbitrary, $\beta_X(L,m)\leq \gamma\beta_X(K,m)$ as claimed.
\end{proof}

\begin{thm}
  Let $X=(\R^3,\norm{\cdot})$, $m\in\PI$, and
  \begin{displaymath}
    \eta=\sup\SET{\beta_X(T,m)}{T\text{ is a simplex in $\R^3$}}.
  \end{displaymath}
  We have
  \begin{displaymath}
    \beta(X,m)\leq\min\limits_{\varepsilon\in(0,1/3)}\max\qty{\qty(1+\frac{4\varepsilon}{1-3\varepsilon})\eta,\frac{2(3-\varepsilon)}{4-\varepsilon}\beta_X(B_X,m)}.
  \end{displaymath}
\end{thm}
\begin{proof}
  Let $K$ be a complete set in $X$, $\varepsilon$ be a number in $\qty(0,\frac{1}{3})$. We distinguish two cases.

  \textbf{Case 1}. The Banach-Mazur distance from $K$ to three-dimensional simplices is bounded from the above by
  \begin{displaymath}
    1+\frac{4\varepsilon}{1-3\varepsilon}.
  \end{displaymath}
  Then there exist a tetrahedron $T$ and a point $c\in\R^3$ such that
  \begin{displaymath}
    T\subset K\subset \qty(1+\frac{4\varepsilon}{1-3\varepsilon}) T+c. 
  \end{displaymath}
  By Lemma \ref{lem:estimation-via-Hausdorff}, we have
  \begin{displaymath}
    \beta_X(K,m)\leq \qty(1+\frac{4\varepsilon}{1-3\varepsilon})\eta.
  \end{displaymath}

  \textbf{Case 2}. The Banach-Mazur distance from $K$ to three-dimensional
  simplex is at least
  \begin{displaymath}
    1+\frac{4\varepsilon}{1-3\varepsilon}.
  \end{displaymath}
  From Theorem 2.1 in \cite{Schneider2009}, it follows that
  \begin{displaymath}
    s(K)\leq 3-\varepsilon.
  \end{displaymath}
  Denote by $R(K)$ the circumradius of $K$. Theorem 1.1 in
  \cite{Brandenberg-Merino2017IJM} shows that
  \begin{displaymath}
    s(K)=\frac{R(K)/\DIAM{K}}{1-R(K)/\DIAM{K}}.
  \end{displaymath}
  It follows that
  \begin{displaymath}
    \frac{R(K)}{\DIAM{K}}\leq \frac{3-\varepsilon}{4-\varepsilon}.
  \end{displaymath}
  By a suitable translation if necessary, we may assume that
  \begin{displaymath}
    K\subseteq\frac{3-\varepsilon}{4-\varepsilon}\DIAM{K} B_X.
  \end{displaymath}
  For each $\gamma>0$, there exists a collection $\SET{B_i}{i\in[8]}$ such that
  \begin{displaymath}
    B_X=\bigcup_{i\in[8]}B_i\qqtext{and} \DIAM{B_i}\leq 2\beta_X(B_X,m)+\gamma,~\forall
    i\in[m].
  \end{displaymath}
  It follows that
  \begin{displaymath}
    \beta_X(K,m)\leq \frac{2(3-\varepsilon)}{4-\varepsilon}\beta_X(B_X,m)+\frac{3-\varepsilon}{4-\varepsilon}\gamma.
  \end{displaymath}
  Hence
  \begin{displaymath}
    \beta_X(K,m)\leq \frac{2(3-\varepsilon)}{4-\varepsilon}\beta_X(B_X,m).
  \end{displaymath}
  This completes the proof.
\end{proof}

\begin{cor}\label{cor:estimation-via-ball}
  Let $X=(\R^3,\norm{\cdot})$. If $\beta_X(B_X,8)<\frac{221}{328}$, then $\beta(X,8)<1$.
\end{cor}
\begin{proof}
  Since $\frac{2(3-\frac{7}{57})}{4-\frac{7}{57}}\frac{221}{328}=1$ and
  \begin{displaymath}
    \frac{2(3-\varepsilon)}{4-\varepsilon}=2-\frac{2}{4-\varepsilon}
  \end{displaymath}
  is continuous with respect to $\varepsilon$ on $(0,\frac{1}{3})$, there exists
  a number $\varepsilon_0<\frac{7}{57}$ such that
  \begin{displaymath}
    \frac{2(3-\varepsilon_0)}{4-\varepsilon_0}\beta_X(B_X,8)<1
  \end{displaymath}
  It follows that
  \begin{displaymath}
    \beta(X,8)\leq\max\qty{\qty(1+\frac{4\varepsilon_0}{1-3\varepsilon_0})\frac{9}{16},\frac{2(3-\varepsilon_0)}{4-\varepsilon_0}\beta_X(B_X,8)}<1.\qedhere
  \end{displaymath}
\end{proof}

In particular, Corollary \ref{cor:estimation-via-ball} shows that
$\beta(l_1^3,8)<1$ since the unit ball of $l_1^3$ can be covered by $8$ balls
having radius $\frac{2}{3}<\frac{221}{328}$. By solving the optimization problem
\begin{displaymath}
  \min_{\varepsilon\in(0,1/3)}\max\qty{\qty(1+\frac{4\varepsilon}{1-3\varepsilon})\frac{9}{16},\frac{2(3-\varepsilon)}{4-\varepsilon}\frac{2}{3}},
\end{displaymath}
one can show that $\beta(l_1^3,8)\leq 0.989\ldots$. This estimation can be
improved, see the next section.

\section{$\beta\qty(l_p^3,8)$}
\label{sec:lp}

\begin{lem}\label{lem:BM-distance}
  For each $p\in[1,2]$, $d_{BM}^M\qty(l_p^3,l_\infty^3)\leq \frac{\sqrt{18\cdot
      19}}{10}\approx 1.85$.
\end{lem}
\begin{proof}
  Put $c_1=(3,3,-2)$, $c_2=(-2,3,3)$, $c_3=(3,-2,3)$. Denote by $Q$ the
  parallelipiped having
  \begin{displaymath}
    \SET{\sum\limits_{i\in[3]}\sigma_ic_i}{\sigma_i\in\qty{-1,1},~\forall i\in[3]}
  \end{displaymath}
  as the set of vertices. We have
  \begin{align*}
    \max\SET{\norm{\sum\limits_{i\in[3]}\sigma_ic_i}_p}{i\in[3]}&=\max\qty{\norm{(4,4,4)}_p,\norm{(-2,8,-2)}_p,\norm{(8,-2,-2)}_p,\norm{(2,2,-8)}_p}\\
                                                                &=\max\qty{\norm{(4,4,4)}_p,\norm{(-2,8,-2)}_p}\\
                                                                &=\norm{(-2,8,-2)}_p=2\norm{(-1,4,-1)}_p=2\norm{(1,1,4)}_p.
  \end{align*}
  It follows that
  \begin{displaymath}
    \frac{1}{2\norm{(1,1,4)}_p}Q\subset B_p^3,
  \end{displaymath}
  where $B_p^3$ is the unit ball of $l_p^3$. Let $q$ be the number satisfying
  \begin{displaymath}
    \frac{1}{p}+\frac{1}{q}=1,
  \end{displaymath}
  and let $f_1$, $f_2$, $f_3$ be linear functionals defined on $l_p^3$ such
  that, for any $(\alpha,\beta,\gamma)\in \R^3$,
  \begin{gather*}
    f_1((\alpha,\beta,\gamma))=\frac{1}{100}\qty(15\alpha-5\beta+15\gamma),\\
    f_2((\alpha,\beta,\gamma))=\frac{1}{100}\qty(-5\alpha+15\beta+15\gamma),\\
    f_3((\alpha,\beta,\gamma))=\frac{1}{100}\qty(15\alpha+15\beta-5\gamma).
  \end{gather*}
  Then
  \begin{gather*}
    c_3+\SPAN\qty{c_1,c_2}=\SET{x\in\R^3}{f_1(x)=1},\\
    c_2+\SPAN\qty{c_1,c_3}=\SET{x\in\R^3}{f_2(x)=1},\\
    c_1+\SPAN\qty{c_2,c_3}=\SET{x\in\R^3}{f_3(x)=1}.
  \end{gather*}
  Thus the distances from the origin $\origin$ to the facets of $Q$ all equals
  to
  \begin{displaymath}
    \frac{100}{\norm{(15,-5,15)}_q}.
  \end{displaymath}
  It follows that
  \begin{displaymath}
    \frac{1}{2\norm{(1,1,4)}_p}Q\subset B_p^3\subset \frac{\norm{(15,-5,15)}_q}{100}Q=\frac{\norm{(1,1,4)}_p\norm{(3,1,3)}_q}{10}\frac{1}{2\norm{(1,1,4)}_p}Q,
  \end{displaymath}
  which implies that
  \begin{displaymath}
    d_{BM}^M(l_p^3,l_\infty^3)\leq
    \frac{\norm{(1,1,4)}_p\norm{(3,1,3)}_q}{10}\leq
    \frac{\norm{(1,1,4)}_2\norm{(3,1,3)}_2}{10}=\frac{\sqrt{18\cdot 19}}{10}.\qedhere
  \end{displaymath}
\end{proof}
\begin{rem}
  The last inequality in Lemma \ref{lem:BM-distance} can be verified in the
  following way. Put
  \begin{displaymath}
    f(p)=\qty(4^p + 2)^{\frac{1}{p}}\cdot \qty(2\cdot 3^{\frac{p}{p-1}} + 1)^{\frac{p-1}{p}}.
  \end{displaymath}
  Numerical results show that $f'(p)=0$ has a unique solution $p_0\approx 1.320$
  in $[1,2]$, and
  \begin{displaymath}
    f(p_0)\approx 17.550<f(2).
  \end{displaymath}
  Moreover, $f(1)<f(2)$. Thus $f(p)$ is maximized at $p=2$.
  \par

  Numerical results show that when $p\in[1,1.736)$,
  \begin{displaymath}
    \frac{\norm{(1,1,4)}_p\norm{(3,1,3)}_q}{10}\leq \frac{9}{5}.
  \end{displaymath}
  The estimation in Lemma \ref{lem:BM-distance} could be improved by
  choosing points $c_1$, $c_2$, $c_3$ more carefully for different $p\in[1,2]$.
\end{rem}

\begin{thm}
We have the following estimation:
\begin{displaymath}
  \beta\qty(l_p^3,8)\leq
  \begin{cases}
    \frac{\sqrt{18\cdot 19}}{20},&p\in[1,2),\\
    \frac{3^{1/p}}{2},&p\in[2,+\infty].
  \end{cases}
\end{displaymath}
\end{thm}
\begin{proof}
  First we consider the case when $p\in[2,+\infty]$. By Proposition 37.6 in
  \cite{Jaegermann1989}, $d_{BM}^M(l_p^3,l_\infty^3)=3^{1/p}$, this together
  with Theorem \ref{thm:stability}, implies that $\beta(l_p^3,8)\leq
  \frac{3^{1/p}}{2}\leq \frac{\sqrt{3}}{2}$.

The case when $p\in[1,2]$ follows directly from Lemma \ref{lem:BM-distance} and
Theorem \ref{thm:stability}.
\end{proof}

\section{Acknowledgements}
\label{sec:ack}
This work is supported by the National Natural Science Foundation of China
(grant numbers: 11921001 and 12071444), the National Key Research and
Development Program of China (2018YFA0704701), and the Natural Science
Foundation of Shanxi Province of China (201901D111141). The authors are grateful
to Professor C. Zong for his supervision and discussion.

\bibliographystyle{amsplain}

\bibliography{../xbib/references}

\end{document}